\documentclass[10pt,a4paper,twoside]{amsart}

\usepackage{graphicx}
\usepackage{a4wide}

\newtheorem{definition}{Definition}[section]
\newtheorem{theorem}{Theorem}[section]
\newtheorem{lemma}{Lemma}[section]

\newtheorem*{maintheorem*}{Main Theorem}
\allowdisplaybreaks
\numberwithin{equation}{section}

\newcommand{\norm}[1]{\left\| #1 \right\|}

\newcommand{\eps}{\varepsilon}

\newcommand{\eb}{{\eps,\beta}}
\newcommand{\ueb}{u_\eb}

\newcommand{\pt}{\partial_t}

\newcommand{\px}{\partial_x }
\newcommand{\pxx}{\partial_{xx}^2}
\newcommand{\pxxx}{\partial_{xxx}^3}

\renewcommand{\i}{\ifmmode\mathit{\mathchar"7010 }\else\char"10 \fi}
\renewcommand{\j}{\ifmmode\mathit{\mathchar"7011 }\else\char"11 \fi}
\newcommand{\R}{\mathbb{R}}
\newcommand{\N}{\mathbb{N}}

\newcommand{\supp}{\mathrm{supp}\,}

{%

\begin{enumerate}}%
{\end{enumerate}}

{%

\begin{enumerate}}%
{\end{enumerate}}

\begin{document}\large

\title[Ibragimov-Shabat equation]{A singular limit problem\\  for the Ibragimov-Shabat equation}

\author[G. M. Coclite and L. di Ruvo]{Giuseppe Maria Coclite and Lorenzo di Ruvo}
\address[Giuseppe Maria Coclite and Lorenzo di Ruvo]
{\newline Department of Mathematics,   University of Bari, via E. Orabona 4, 70125 Bari,   Italy}
\email[]{giuseppemaria.coclite@uniba.it, lorenzo.diruvo@uniba.it}
\urladdr{http://www.dm.uniba.it/Members/coclitegm/}

\date{\today}

\keywords{Singular limit, compensated compactness, Ibragimov-Shabat equation, entropy condition.}

\subjclass[2000]{35G25, 35L65, 35L05}


\thanks{The authors are members of the Gruppo Nazionale per l'Analisi Matematica, la Probabilit\`a e le loro Applicazioni (GNAMPA) of the Istituto Nazionale di Alta Matematica (INdAM)}

\begin{abstract}
We consider the Ibragimov-Shabat equation, which contains nonlinear dispersive effects. We prove
that as the diffusion parameter tends to zero, the solutions of the dispersive equation converge to discontinuous weak solutions of a scalar conservation law. The proof relies on deriving suitable a priori estimates together with an application of the compensated compactness method in the $L^p$ setting.
\end{abstract}

\maketitle


\section{Introduction}
\label{sec:intro}

B\"acklund transformations have been useful in the calculation of soliton solutions of certain nonlinear evolution
equations of physical significance \cite{BD, La, Rs1, Rs2} restricted to one space variable $x$ and a time coordinate $t$.
The classical treatment of the surface transformations, which provide the origin of B\"acklund theory, was developed in \cite{G}.
B\"acklund transformations are local geometric transformations, which construct from a given surface of constant Gaussian curvature $-1$
a two parameter family of such surfaces. To find such transformations, one needs to solve a system of compatible
ordinary differential equations \cite{ET}.

In \cite{KCS, KCS1}, the authors used the notion of  differential equation for a function $u(t,x)$ that describes a pseudo-spherical
surface, and they derived some B\"acklund transformations for  nonlinear evolution equations which are the integrability condition $sl(2,R)-$ valued linear problems \cite{CKI, ACKS,EHK, CKSS,Rs2}.

In \cite{KW}, the authors had derived some B\"acklund transformations for nonlinear evolution
equations of the AKNS class. These transformations explicitly express the new solutions in terms of the known solutions of the  nonlinear evolution
equations and corresponding wave functions which are solutions of the associated Ablowitz-Kaup-Newell-Segur (AKNS) system \cite{AKNS,ZS}.

In \cite{KCS2}, the authors used B\"acklund transformations derived in \cite{KCS,KCS1} in the construction of exact soliton solutions for some nonlinear evolution equations describing pseudo-spherical surfaces which are beyond the AKNS class. In particular, they analyzed the following equation \cite{BRT}:
\begin{equation}
\label{eq:ra1}
\px\left(\pt u +\alpha  g(u)\px u +\beta\px u\right)=\gamma g'(u), \quad \alpha,\,\beta,\,\gamma\in\R,
\end{equation}
where $g(u)$ is any solution of the linear ordinary differential equation
\begin{equation}
\label{eq:ra2}
g''(u)+\mu g(u)=\theta, \quad \mu,\,\theta\in\R.
\end{equation}
\eqref{eq:ra1} include the sine-Gordon, sinh-Gordon and Liouville equations, in correspondence of $\alpha=0$.

In \cite{RA}, Rabelo proved that the system of the equations \eqref{eq:ra1} and \eqref{eq:ra2} describes pseudo-spherical surfaces and possesses a zero-curvature representation with a parameter.

In \cite{Cd1}, the authors investigated the well-posedness in classes of discontinuous functions of \eqref{eq:ra1}, when $\alpha=-1,\, \beta=0,\, \mu=0,\,  \theta=1$.

Moreover, in \cite{Cd2}, the authors investigated the well-posedness in classes of discontinuous functions of \eqref{eq:ra1}, when $\alpha=1,\, \beta=0,\, \mu=-1,\, \theta=1,\, \gamma=-1$.

One more equation, that describes pseudo-spherical surface, is the following one \cite{SEG}:
\begin{equation}
\label{eq:Ise10}
\pt \ueb-\frac{3}{5}\px (\ueb^5)+\beta\eps\pxxx u=3\beta\eps\ueb^2\pxx\ueb-9\ueb(\px\ueb)^2,
\end{equation}
which is the Ibraginov-Shabat equation.
Following \cite{CdREM,CK}, we consider the following diffusive approximation of \eqref{eq:Ise10}
\begin{equation}
\label{eq:Ise}
\pt \ueb-\frac{3}{5}\px (\ueb^5)+\beta \pxxx u=3\beta\ueb^2\pxx\ueb-9\beta\ueb(\px\ueb)^2 +\eps\pxx\ueb,
\end{equation}
where $\beta$ is the dispersive parameter.

We consider the initial value problem for \eqref{eq:Ise}, so we augment \eqref{eq:Ise} with the
initial condition
\begin{equation}
u(0,x)=u_{0}(x),
\end{equation}
on which we assume that
\begin{equation}
\label{eq:u-0}
u_{0}\in L^2(\R)\cap L^{10}(\R).
\end{equation}
We are interested in the no high frequency limit, i.e., we send $\beta\to 0$ in \eqref{eq:Ise}. In this way, we pass from \eqref{eq:Ise} to
\begin{equation}
\label{eq:con-law}
\begin{cases}
\pt u-\frac{3}{5}\px (u^5)=0, \quad& t>0,\,x\in\R,\\
u(0,x)=u_{0}(x),\quad & x\in\R,
\end{cases}
\end{equation}
which is a scalar conservation law.

We study the dispersion-diffusion limit for \eqref{eq:Ise}. Therefore, we fixe two small numbers $0<\eps,\,\beta <1$, and consider the following third order problem
\begin{equation}
\label{eq:Iseepsw}
\begin{cases}
\displaystyle\pt \ueb -\frac{3}{5}\px \ueb^5 +\beta\pxxx\ueb\\
\quad\quad=  3\beta\ueb^2\pxx\ueb- 9\beta\ueb(\px\ueb)^2+\eps\pxx\ueb,&\quad t>0,\ x\in\R,\\
\ueb(0,x)=u_{\eps,\beta,0}(x),&\quad x\in\R,
\end{cases}
\end{equation}
where $u_{\eps,\beta,0}$ is a $C^\infty$ approximation of $u_{0}$ such that
\begin{equation}
\begin{split}
\label{eq:u0eps}
&u_{\eps,\,\beta,\,0} \to u_{0} \quad  \textrm{in $L^{p}_{loc}(\R)$, $1\le p < 10$, as $\eps,\,\beta \to 0$,}\\
&\norm{u_{\eps,\beta, 0}}^2_{L^2(\R)}+ \norm{u_{\eps,\beta, 0}}^{10}_{L^{10}(\R)}+(\beta+ \eps^2) \norm{\px u_{\eps,\beta,0}}^2_{L^2(\R)}\le C_{0}, \quad \eps,\beta >0,
\end{split}
\end{equation}
where $C_0$ is a constant independent on $\eps$ and $\beta$.

The main result of this paper is the following theorem.

\begin{theorem}
\label{th:main}
Assume that \eqref{eq:u-0} and  \eqref{eq:u0eps} hold.
If
\begin{equation}
\label{eq:beta-eps}
\beta=\mathbf{\mathcal{O}}(\eps^2),
\end{equation}
then, there exist two sequences $\{\eps_{n}\}_{n\in\N}$, $\{\beta_{n}\}_{n\in\N}$, with $\eps_n, \beta_n \to 0$, and a limit function
\begin{equation*}
u\in L^{\infty}(0,T; L^2(\R)\cap L^{10}(\R)), \quad T>0,\\
\end{equation*}
such that
\begin{itemize}
\item[$i)$] $u$ is  a distributional solution of \eqref{eq:con-law},
\item[$ii)$] $u_{\eps_n, \beta_n}\to u$ strongly in $L^{p}_{loc}((0,\infty)\times\R)$, for each $1\le p <10$.
\end{itemize}
Moreover, if
\begin{equation}
\label{eq:beta-eps-1}
\beta=\mathbf{\mathcal{O}}(\eps^{2+\alpha}), \quad \textrm{for some $\alpha>0$},
\end{equation}
then,
\begin{itemize}
\item[$iii)$] $u$ is  the unique entropy solution of \eqref{eq:con-law}.
\end{itemize}
\end{theorem}
The paper is organized in three sections. In Section \ref{sec:vv}, we prove some a priori estimates, while in Section \ref{sec:theor} we prove Theorem \ref{th:main}.

\section{A priori Estimates}
\label{sec:vv}

This section is devoted to some a priori estimates on $\ueb$. We denote with $C_0$ the constants which depend only on the initial data, and with $C(T)$ the constants which depend also on $T$.

\begin{lemma}\label{lm:l-2}
For each $t>0$,
\begin{equation}
\label{eq:l-2}
\norm{\ueb(t,\cdot)}^2_{L^2(\R)}+2\eps\int_{0}^{t}\norm{\px\ueb(s,\cdot)}^2_{L^2(\R)}ds +36\beta\int_{0}^{t}\!\!\!\int_{\R}\ueb^2(\px\ueb)^2dsdx\le C_{0}.
\end{equation}
\end{lemma}
\begin{proof}
Multiplying \eqref{eq:Iseepsw} by $\ueb$, we have
\begin{align*}
\frac{1}{2}\frac{d}{dt}\int_{\R}\ueb^2 dx=&\int_{\R}\ueb\pt\ueb dx\\
=& 3\int_{\R}\ueb^5\px\ueb dx - \beta\int_{\R}\ueb\pxxx\ueb dx \\
&+3\beta\int_{\R}\ueb^3\pxx\ueb dx - 9\beta \int_{\R}\ueb^2(\px\ueb)^2 dx\\
&+\eps\int_{\R}\ueb\pxx\ueb dx\\
=& \beta \int_{\R}\px\ueb\pxx\ueb dx -9\beta\int_{\R}\ueb^2(\px\ueb)^2 dx\\
&- 9\beta \int_{\R}\ueb^2(\px\ueb)^2 dx-\eps\int_{\R}(\px\ueb)^2dx,
\end{align*}
that is
\begin{equation}
\frac{d}{dt}\norm{\ueb(t,\cdot)}^2_{L^2(\R)}2\eps\norm{\px\ueb(t,\cdot)}^2_{L^2(\R)}+18\beta\int_{\R}\ueb^2(\px\ueb)^2 dx=0.
\end{equation}
An integration on $(0,t)$ and \eqref{eq:u0eps} give \eqref{eq:l-2}.
\end{proof}

\begin{lemma}\label{lm:u-infty}
For each $t>0$,
\begin{equation}
\label{eq:pxu}
\begin{split}
\beta\norm{\px\ueb(t,\cdot)}^2_{L^2(\R)} &+\frac{1}{5}\norm{\ueb(t,\cdot)}^6_{L^{6}(\R)}\\
& +2\beta\eps\int_{0}^{t}\norm{\pxx\ueb(s,\cdot)}^2_{L^2(\R)}ds +6\beta^2\int_{0}^{t}\norm{\px\ueb(s,\cdot)}^4_{L^4(\R)}ds\\
& +6\eps\int_{0}^{t}\!\!\!\int_{\R}\ueb^4(\px\ueb)^2 dsdx + 6\beta^2\int_{0}^{t}\!\!\!\int_{\R}\ueb^2(\pxx\ueb)^2 ds dx\\
&  + \frac{42}{5}\beta\int_{0}^{t}\!\!\!\int_{\R}\ueb^6(\px\ueb)^2dx \le C_{0}.
\end{split}
\end{equation}
In particular, we have
\begin{equation}
\label{eq:u-l-infty}
\norm{\ueb(t,\cdot)}_{L^{\infty}(\R)}\le C_{0}\beta^{-\frac{1}{4}}.
\end{equation}
Moreover, fixed $T>0$,
\begin{align}
\label{eq:u-l-6}
\norm{\ueb}_{L^6((0,T)\times\R)}\le&C_{0}T^{\frac{1}{6}},\\
\label{eq:u-l-4}
\norm{\ueb}_{L^4((0,T)\times\R))}\le& C_{0}T^{\frac{1}{4}}.
\end{align}
\end{lemma}
\begin{proof}
Multiplying  \eqref{eq:Iseepsw} by $\displaystyle -\beta\pxx\ueb +\frac{3}{5}\ueb^5$,  we have
\begin{equation}
\label{eq:Ohmp}
\begin{split}
\left(-\beta\pxx\ueb +\frac{3}{5}\ueb^5\right)\pt\ueb &- \frac{3}{5}\left(-\beta\pxx\ueb +\frac{3}{5}\ueb^5\right)\px\ueb^5\\
&+\beta\left(-\beta\pxx\ueb +\frac{3}{5}\ueb^5\right)\pxxx\ueb\\
=&3\beta\left(-\beta\pxx\ueb +\frac{3}{5} \ueb^5\right)\ueb^2\pxx\ueb\\
&-9\beta \left(-\beta\pxx\ueb +\frac{3}{5} \ueb^5\right)\ueb(\px\ueb)^2\\
&+\eps\left(-\beta\pxx\ueb+ \frac{3}{5}\ueb^5\right)\pxx\ueb.
\end{split}
\end{equation}
Observe that
\begin{align*}
\int_{\R}\left(-\beta\pxx\ueb+\frac{3}{5}\ueb^5\right)\pt\ueb dx=&\frac{d}{dt}\left(\frac{\beta}{2}\norm{\px\ueb(t,\cdot)}^2_{L^2(\R)}+\frac{1}{10}\norm{\ueb(t,\cdot)}^6_{L^{6}(\R)}\right),\\
- \frac{3}{5}\int_{\R}\left(-\beta\pxx\ueb +\frac{3}{5}\ueb^5\right)\px \ueb^5 dx =&3\beta\int_{\R}\ueb^4\px\ueb\pxx\ueb dx,\\
\beta\int_{\R}\left(-\beta\pxx\ueb +\frac{3}{5}\ueb^5\right)\pxxx\ueb dx =& -3\beta\int_{\R}\ueb^4\px\ueb\pxx\ueb dx,\\
3\beta\int_{\R}\left(-\beta\pxx\ueb +\frac{3}{5} \ueb^5\right)\ueb^2\pxx\ueb dx=&-3\beta^2\int_{\R}\ueb^2(\pxx\ueb)^2 dx\\
&+ \frac{9}{5} \beta \int_{\R} \ueb^{7}\pxx\ueb dx, \\
-9\beta\int_{\R} \left(-\beta\pxx\ueb +\frac{3}{5}\ueb^5\right)\ueb(\px\ueb)^2 =&9\beta^2\int_{\R}\ueb(\px\ueb)^2 \pxx\ueb dx\\
& -\frac{27}{5}\beta\int_{\R}\ueb^6(\px\ueb)^2dx,\\
\eps\int_{\R} \left(-\beta\pxx\ueb +\frac{3}{5}\ueb^5\right)\pxx\ueb dx = & -\eps\beta \norm{\pxx\ueb(t,\cdot)}^2_{L^2(\R)}\\
&- 3\eps\int_{\R}\ueb^4(\px\ueb)^2dx.
\end{align*}
Therefore, integrating \eqref{eq:Ohmp} over $\R$,
\begin{equation}
\label{eq:p12}
\begin{split}
&\frac{d}{dt}\left(\frac{\beta}{2}\norm{\px\ueb(t,\cdot)}^2_{L^2(\R)}+\frac{1}{10}\norm{\ueb(t,\cdot)}^6_{L^{6}(\R)}\right)\\
&\qquad +\beta\eps\norm{\pxx\ueb(t,\cdot)}^2_{L^2(\R)} +3\eps\int_{\R}\ueb^4(\px\ueb)^2dx\\
 &\qquad+ 3\beta^2\int_{\R}\ueb^2(\pxx\ueb)^2 dx + \frac{27}{5}\beta\int_{\R}\ueb^6(\px\ueb)^2dx\\
&\quad = \frac{9}{5} \beta \int_{\R} \ueb^{7}\pxx\ueb dx +9\beta^2\int_{\R}\ueb(\px\ueb)^2 \pxx\ueb dx.
\end{split}
\end{equation}
Since
\begin{align*}
\frac{9}{5} \beta \int_{\R} \ueb^{7}\pxx\ueb dx=&-\frac{63}{5}\beta\int_{\R}\ueb^6(\px\ueb)^2 dx,\\
9\beta^2\int_{\R}\ueb(\px\ueb)^2 \pxx\ueb dx=& -3\beta^2\int_{\R}(\px\ueb)^4dx ,
\end{align*}
it follows from \eqref{eq:p12} that
\begin{align*}
&\frac{d}{dt}\left(\frac{\beta}{2}\norm{\px\ueb(t,\cdot)}^2_{L^2(\R)}+\frac{1}{10}\norm{\ueb(t,\cdot)}^6_{L^{6}(\R)}\right)\\
&\qquad + \beta\eps\norm{\pxx\ueb(t,\cdot)}^2_{L^2(\R)} +3\eps\int_{\R}\ueb^4(\px\ueb)^2dx\\
&\qquad + 3\beta^2\int_{\R}\ueb^2(\pxx\ueb)^2 dx + 18\beta\int_{\R}\ueb^6(\px\ueb)^2dx \\
&\qquad +3\beta^2\norm{\px\ueb(t,\cdot)}^4_{L^4(\R)}=0.
\end{align*}
An integration on $(0,t)$ and \eqref{eq:u0eps} give \eqref{eq:pxu}.

We prove \eqref{eq:u-l-infty}. Due to \eqref{eq:l-2}, \eqref{eq:pxu} and the H\"older inequality,
\begin{align*}
\ueb^2(t,x)=&2\int_{-\infty}^x \ueb\px\ueb dx \le 2\int_{\R}\vert \ueb\vert \vert \px\ueb\vert dx\\
\le&2\norm{\ueb(t,\cdot)}_{L^2(\R)}\norm{\px\ueb(t,\cdot)}_{L^2(\R)}\le C_{0}\beta^{-\frac{1}{2}},
\end{align*}
which gives \eqref{eq:u-l-infty}.

We prove \eqref{eq:u-l-6}. From \eqref{lm:u-infty}, we have
\begin{equation*}
\norm{\ueb(t,\cdot)}^6_{L^6(\R)}\le C_{0}.
\end{equation*}
An integration on $(0,T)$ gives \eqref{eq:u-l-6}.

Finally, we prove \eqref{eq:u-l-4}. Due to \eqref{eq:l-2}, \eqref{eq:pxu} and the Young inequality,
\begin{equation}
\label{eq:12346}
\int_{\R}\ueb^4 dx \le \frac{1}{2}\int_{\R}\ueb^2 dx + \frac{1}{2}\int_{\R}\ueb^6 dx \le C_{0}.
\end{equation}
Therefore, fix $T>0$, \eqref{eq:u-l-4} follows from \eqref{eq:12346} and an integration on $(0,T)$.
\end{proof}

\begin{lemma}\label{lm:stima-l-10}
Let  $T>0$. Assume \eqref{eq:beta-eps} holds true. There exists  $C(T)>0$, independent on $\eps$ and $\beta$, such that
\begin{equation}
\label{eq:l-10}
\begin{split}
\frac{1}{10}\norm{\ueb(t,\cdot)}^{10}_{L^{10}(\R)}&+\frac{3\eps^2}{2}\norm{\px\ueb(t,\cdot)}^2_{L^2(\R)}\\
& +45\eps^2\beta\int_{0}^{t}\norm{\px\ueb(s,\cdot)}^{4}_{L^4(\R)}ds+4\eps\int_{0}^{t}\!\!\!\int_{\R}\ueb^8(\px\ueb)^2 dsdx\\
& +\frac{1}{2}\eps^3\int_{0}^{t}\norm{\pxx\ueb(s\cdot)}^2_{L^2(\R)}ds +15\beta\eps^2\int_{0}^{t}\!\!\!\int_{\R}\ueb^2(\pxx\ueb)^2dx\\
& +48\beta\int_{0}^{t}\!\!\!\int_{\R}\ueb^{10}(\px\ueb)^2 dx \le C(T).
\end{split}
\end{equation}
for every $0<t<T$. Moreover,
\begin{align}
\label{eq:00010}
\beta\norm{\px\ueb\pxx\ueb}_{L^{1}((0,T)\times\R)} \le &C(T),\\
\label{eq:defuxx}
\beta^2\int_{0}^{T}\norm{\pxx\ueb(s,\cdot)}^2_{L^2(\R)}ds \le& C(T)\eps.
\end{align}
\end{lemma}
\begin{proof}
Let $0<t< T$. Multiplying \eqref{eq:Iseepsw} by $\ueb^9-5\eps^2\pxx\ueb$, we have
\begin{equation}
\label{eq:Ohmp-1}
\begin{split}
(\ueb^9 -5\eps^2\pxx\ueb)\pt\ueb &- \frac{3}{5}(\ueb^9 -5\eps^2\pxx\ueb)\px \ueb^5\\
&+\beta(\ueb^9 -5\eps^2\pxx\ueb)\pxxx\ueb\\
=& 3\beta(\ueb^9 -5\eps^2\pxx\ueb)\ueb^2\pxx\ueb\\
&-9\beta (\ueb^9 -5\eps^2\pxx\ueb)\ueb(\px\ueb)^2\\
&+ \eps (\ueb^9 -5\eps^2\pxx\ueb)\pxx\ueb.
\end{split}
\end{equation}
Since
\begin{align*}
\int_{\R}(\ueb^9 -3\eps^2\pxx\ueb)\pt\ueb dx=&\frac{d}{dt}\left(\frac{1}{10}\norm{\ueb(t,\cdot)}^{10}_{L^{10}(\R)}+\frac{5\eps^2}{2}\norm{\px\ueb(t,\cdot)}^2_{L^2(\R)}\right),\\
- \frac{3}{5}\int_{\R}(\ueb^9 -5\eps^2\pxx\ueb)\px \ueb^5 dx=&9\eps^2\int_{\R}\ueb^4\px\ueb\pxx\ueb dx,\\
\beta\int_{\R}(\ueb^9 -5\eps^2\pxx\ueb)\pxxx\ueb dx=& -9\beta\int_{\R}\ueb^8\px\ueb\pxx\ueb dx\\
=& 36\beta\int_{\R}\ueb^7(\px\ueb)^2dx,\\
3\beta\int_{\R}(\ueb^9 -5\eps^2\pxx\ueb)\ueb^2\pxx\ueb dx=& 3\beta\int_{\R}\ueb^{11}\pxx\ueb dx \\
&-15\beta\eps^2\int_{\R}\ueb^2(\pxx\ueb)^2dx\\
-9\beta\int_{\R}(\ueb^9 -5\eps^2\pxx\ueb)\ueb(\px\ueb)^2= &-9\beta\int_{\R}\ueb^{10}(\px\ueb)^2 dx\\
&+45\eps^2\beta\int_{\R}\ueb(\px\ueb)^2\pxx\ueb dx\\
\eps\int_{\R}(\ueb^9 -5\eps^2\pxx\ueb)\pxx\ueb dx = & -9\eps\int_{\R}\ueb^8(\px\ueb)^2 dx -5\eps^3\norm{\pxx\ueb(t,\cdot)}^2_{L^2(\R)},
\end{align*}
integrating \eqref{eq:Ohmp-1} over $\R$,
\begin{equation}
\label{eq:Ise1}
\begin{split}
&\frac{d}{dt}\left(\frac{1}{10}\norm{\ueb(t,\cdot)}^{10}_{L^{10}(\R)}+\frac{5\eps^2}{2}\norm{\px\ueb(t,\cdot)}^2_{L^2(\R)}\right)\\
&\qquad + 9\eps\int_{\R}\ueb^8(\px\ueb)^2 dx + 5\eps^3\norm{\pxx\ueb(t\cdot)}^2_{L^2(\R)}\\
&\qquad +15\beta\eps^2\int_{\R}\ueb^2(\pxx\ueb)^2dx +15\beta\eps^2\int_{\R}\ueb^{10}(\px\ueb)^2 dx\\
&\quad = 45\eps^2\beta\int_{\R}\ueb(\px\ueb)^2\pxx\ueb dx + 3\beta\int_{\R}\ueb^{11}\pxx\ueb dx\\
&\qquad -9\eps^2\int_{\R}\ueb^4\px\ueb\pxx\ueb dx +36\beta\int_{\R}\ueb^7(\px\ueb)^2 dx.
\end{split}
\end{equation}
Since
\begin{align*}
45\eps^2\beta\int_{\R}\ueb(\px\ueb)^2\pxx\ueb dx =& -45\eps^2\beta\int_{\R}(\px\ueb)^4 dx,\\
3\beta\int_{\R}\ueb^{11}\pxx\ueb dx=& -33\beta\int_{\R}\ueb^{10}(\px\ueb)^2 dx,
\end{align*}
from  \eqref{eq:Ise1}, we have
\begin{equation}
\label{eq:Ise3}
\begin{split}
&\frac{d}{dt}\left(\frac{1}{10}\norm{\ueb(t,\cdot)}^{10}_{L^{10}(\R)}+\frac{3\eps^2}{2}\norm{\px\ueb(t,\cdot)}^2_{L^2(\R)}\right)\\
&\qquad +45\eps^2\beta\norm{\px\ueb(t,\cdot)}^{4}_{L^4(\R)} + 9\eps\int_{\R}\ueb^8(\px\ueb)^2 dx \\
&\qquad +5\eps^3\norm{\pxx\ueb(t\cdot)}^2_{L^2(\R)} +15\beta\eps^2\int_{\R}\ueb^2(\pxx\ueb)^2dx\\
&\qquad +48\beta \int_{\R}\ueb^{10}(\px\ueb)^2 dx\\
&\quad =-9\eps^2\int_{\R}\ueb^4\px\ueb\pxx\ueb dx +36\beta\int_{\R}\ueb^7(\px\ueb)^2 dx.
\end{split}
\end{equation}
Due to the Young inequality,
\begin{equation}
\label{eq:you10}
\begin{split}
&9\eps^2\left\vert\int_{\R}\ueb^4\px\ueb\pxx\ueb dx\right\vert \le 9\int_{\R}\left\vert\eps^{\frac{1}{2}}\ueb^4\px\ueb\right\vert\left\vert\eps^{\frac{3}{2}}\pxx\ueb\right\vert dx\\
&\quad \le\frac{9\eps}{2}\int_{\R}\ueb^8(\px\ueb)^2dx +\frac{9\eps^3}{2}\norm{\pxx\ueb(t,\cdot)}^2_{L^2(\R)}.
\end{split}
\end{equation}
It follows from \eqref{eq:Ise3} and \eqref{eq:you10} that
\begin{equation}
\label{eq:Ise4}
\begin{split}
&\frac{d}{dt}\left(\frac{1}{10}\norm{\ueb(t,\cdot)}^{10}_{L^{10}(\R)}+\frac{3\eps^2}{2}\norm{\px\ueb(t,\cdot)}^2_{L^2(\R)}\right)\\
&\qquad +45\eps^2\beta\norm{\px\ueb(t,\cdot)}^{4}_{L^4(\R)} + \frac{9}{2}\eps\int_{\R}\ueb^8(\px\ueb)^2 dx \\
&\qquad +\frac{1}{2}\eps^3\norm{\pxx\ueb(t\cdot)}^2_{L^2(\R)} +15\beta\eps^2\int_{\R}\ueb^2(\pxx\ueb)^2dx\\
&\qquad +48\beta \int_{\R}\ueb^{10}(\px\ueb)^2 dx=36\beta\int_{\R}\ueb^7(\px\ueb)^2 dx.
\end{split}
\end{equation}
Since $0<\eps<1$, thanks to \eqref{eq:beta-eps}, \eqref{eq:u-l-infty} and the Young inequality,
\begin{equation}
\label{eq:you15}
\begin{split}
&36\beta\left\vert\int_{\R}\ueb^7(\px\ueb)^2 dx\right\vert\le 36\beta\int_{\R}\vert\ueb\vert^7 (\px\ueb)^2 dx \\
&\quad \le 36\beta \norm{\ueb(t,\cdot)}_{L^{\infty}(\R)}\int_{\R}\ueb^6(\px\ueb)^2 dx\le C_{0}\beta^{\frac{3}{4}}\int_{\R}\ueb^6(\px\ueb)^2 dx\\
&\quad\le\eps^{\frac{3}{2}}\int_{\R}\vert C_{0}\ueb^2\px\ueb\vert\ueb^4\vert\px\ueb\vert dx\\
&\quad= \int_{\R}\vert C_{0}\eps\ueb^2\px\ueb\left\vert\eps^{\frac{1}{2}}\ueb^4\vert\px\ueb\right\vert dx \\
&\quad\le \eps^2 C_{0}\int_{\R}\ueb^4(\px\ueb)^2dx + \frac{\eps}{2}\int_{\R}\ueb^8(\px\ueb)^2 dx\\
&\quad \le \eps C_0\int_{\R}\ueb^4(\px\ueb)^2dx  + \frac{\eps}{2}\int_{\R}\ueb^8(\px\ueb)^2 dx.
\end{split}
\end{equation}
Therefore, from \eqref{eq:Ise4} and \eqref{eq:you15},
\begin{equation*}
\label{eq:Ise5}
\begin{split}
&\frac{d}{dt}\left(\frac{1}{10}\norm{\ueb(t,\cdot)}^{10}_{L^{10}(\R)}+\frac{3\eps^2}{2}\norm{\px\ueb(t,\cdot)}^2_{L^2(\R)}\right)\\
&\qquad +45\eps^2\beta\norm{\px\ueb(t,\cdot)}^{4}_{L^4(\R)} + 4\eps\int_{\R}\ueb^8(\px\ueb)^2 dx \\
&\qquad +\frac{1}{2}\eps^3\norm{\pxx\ueb(t\cdot)}^2_{L^2(\R)} +15\beta\eps^2\int_{\R}\ueb^2(\pxx\ueb)^2dx\\
&\qquad +48\beta \int_{\R}\ueb^{10}(\px\ueb)^2 dx\le \eps C_0\int_{\R}\ueb^4(\px\ueb)^2dx.
\end{split}
\end{equation*}
An integration on $(0,t)$ and \eqref{eq:pxu} give \eqref{eq:l-10}.

We show that \eqref{eq:defuxx} holds. Thanks to  \eqref{eq:beta-eps}, \eqref{eq:l-2}, \eqref{eq:l-10} and H\"older inequality,
\begin{align*}
&\beta\int_{0}^{T}\!\!\!\int_{\R}\vert\px\ueb\pxx\ueb\vert dsdx =\frac{\beta}{\eps^2}\int_{0}^{T}\!\!\!\int_{\R}\eps^{\frac{1}{2}}\vert\px\ueb\vert\eps^{\frac{3}{2}}\vert\pxx\ueb\vert dx\\
&\quad \le \frac{\beta}{\eps^2} \left(\eps \int_{0}^{T}\!\!\!\int_{\R}(\px\ueb)^2 dsdx\right)^{\frac{1}{2}}\left(\eps^3 \int_{0}^{T}\!\!\!\int_{\R}(\pxx\ueb)^2 dsdx\right)^{\frac{1}{2}}\\
&\quad \le C(T)\frac{\beta}{\eps^2}\le C(T),
\end{align*}
that is \eqref{eq:00010}.

Finally, we prove \eqref{eq:defuxx}. Due to \eqref{eq:beta-eps} and \eqref{eq:l-10}, we have
\begin{align*}
\beta^2\int_{0}^{T}\norm{\pxx\ueb(s,\cdot)}^2_{L^2(\R)}ds \le C_{0}^2\eps^4\int_{0}^{T}\norm{\pxx\ueb(s,\cdot)}^2_{L^2(\R)}ds\le C(T)\eps,
\end{align*}
which gives \eqref{eq:defuxx}.
\end{proof}

\section{Proof of Theorem \ref{th:main}}
\label{sec:theor}
In this section, we prove Theorem \ref{th:main}. The following technical lemma is needed  \cite{Murat:Hneg}.
\begin{lemma}
\label{lm:1}
Let $\Omega$ be a bounded open subset of $
\R^2$. Suppose that the sequence $\{\mathcal
L_{n}\}_{n\in\mathbb{N}}$ of distributions is bounded in
$W^{-1,\infty}(\Omega)$. Suppose also that
\begin{equation*}
\mathcal L_{n}=\mathcal L_{1,n}+\mathcal L_{2,n},
\end{equation*}
where $\{\mathcal L_{1,n}\}_{n\in\mathbb{N}}$ lies in a
compact subset of $H^{-1}_{loc}(\Omega)$ and
$\{\mathcal L_{2,n}\}_{n\in\mathbb{N}}$ lies in a
bounded subset of $\mathcal{M}_{loc}(\Omega)$. Then $\{\mathcal
L_{n}\}_{n\in\mathbb{N}}$ lies in a compact subset of $H^{-1}_{loc}(\Omega)$.
\end{lemma}
Moreover, we consider the following definition.
\begin{definition}
A pair of functions $(\eta, q)$ is called an  entropy--entropy flux pair if $\eta :\R\to\R$ is a $C^2$ function and $q :\R\to\R$ is defined by
\begin{equation*}
q(u)=-\int_{0}^{u} 3\xi^4\eta'(\xi) d\xi.
\end{equation*}
An entropy-entropy flux pair $(\eta,\, q)$ is called  convex/compactly supported if, in addition, $\eta$ is convex/compactly supported.
\end{definition}
We begin by proving the following result.
\begin{lemma}\label{lm:dist-solution}
Assume that \eqref{eq:u-0}, \eqref{eq:u0eps}, and \eqref{eq:beta-eps} hold. Then for any compactly supported entropy--entropy flux pair $(\eta,\, q)$, there exist two sequences $\{\eps_{n}\}_{n\in\N}$, $\{\beta_{n}\}_{n\in\N}$, with $\eps_n, \beta_n \to 0$, and a limit function
\begin{equation*}
u\in L^{\infty}(0,T; L^2(\R)\cap L^{10}(\R)),
\end{equation*}
 such that
\begin{equation}
\label{eq:con1}
 u_{\eps_n, \beta_n}\to u \quad  \textrm{in} \quad  L^{p}_{loc}((0,\infty)\times\R),\quad \textrm{for each} \quad 1\le p <10
\end{equation}
and $u$ is a distributional solution of \eqref{eq:Ise}.
\end{lemma}
\begin{proof}

Let $\R^{+} =(0,\infty)$, and let us consider a compactly supported entropy--entropy flux pair $(\eta, q)$. Multiplying \eqref{eq:Iseepsw} by $\eta'(\ueb)$, we have
\begin{align*}
\pt\eta(\ueb) + \px q(\ueb) =&\eps \eta'(\ueb) \pxx\ueb - \beta \eta'(\ueb) \pxxx\ueb\\
&+3\beta\eta'(\ueb)\ueb^2\pxx\ueb -9\beta\eta'(\ueb)\ueb(\px\ueb)^2\\
=& I_{1,\,\eps,\,\beta}+I_{2,\,\eps,\,\beta}+ I_{3,\,\eps,\,\beta} + I_{4,\,\eps,\,\beta}\\
&+I_{5,\,\eps,\,\beta}+I_{6,\,\eps,\,\beta}+I_{7,\,\eps,\,\beta},
\end{align*}
where
\begin{equation}
\begin{split}
\label{eq:12000}
I_{1,\,\eps,\,\beta}&=\px(\eps\eta'(\ueb)\px\ueb),\\
I_{2,\,\eps,\,\beta}&= -\eps\eta''(\ueb)(\px\ueb)^2,\\
I_{3,\,\eps,\,\beta}&= -\px(\beta\eta'(\ueb)\pxx\ueb),\\
I_{4,\,\eps,\,\beta}&= \beta\eta''(\ueb)\px\ueb\pxx\ueb,\\
I_{5,\,\eps,\,\beta}&=\px(3\beta\eta'(\ueb)\ueb^2\px\ueb), \\
I_{6,\,\eps,\,\beta}&=-3\beta\eta''(\ueb)\ueb^2(\px\ueb)^2,\\
I_{7,\,\eps,\,\beta}&=-15\beta\eta'(\ueb)\ueb(\px\ueb)^2.
\end{split}
\end{equation}
We have
\begin{equation*}
\label{eq:H1}
I_{1,\,\eps,\,\beta}\to0 \quad \text{in $H^{-1}((0,T) \times\R),\,T>0$, as $\eps\to 0$.}
\end{equation*}
Thanks to Lemma \ref{lm:l-2},
\begin{align*}
\norm{\eps\eta'(\ueb)\px\ueb}^2_{L^2((0,T)\times\R))}&\leq \norm{\eta'}^2_{L^{\infty}(\R)}\eps ^2\int_{0}^{T} \norm{\px\ueb(s,\cdot)}^2_{L^2(\R)}ds\\
&\leq \norm{\eta'}^2_{L^{\infty}(\R)}\frac{\eps C_{0}}{2} \to 0.
\end{align*}
We claim that
\begin{equation*}
\{I_{2,\,\eps,\,\beta}\}_{\eps,\,\beta >0} \quad\text{is bounded in $L^1((0,T)\times\R),\, T>0$}.
\end{equation*}
Again by Lemma \ref{lm:l-2},
\begin{align*}
\norm{ \eps\eta''(\ueb)(\px\ueb)^2}_{L^1((0,T)\times\R)}& \leq \norm{\eta''}_{L^\infty(\R)}\eps\int_{0}^{T}\norm{\px\ueb(s,\cdot)}^2_{L^2(\R)}ds\\
&\leq\norm{\eta''}_{L^\infty (\R)}\frac{C_0}{2}.
\end{align*}
We have that
\begin{equation*}
I_{3,\,\eps,\,\beta}\to0 \quad \text{in $H^{-1}((0,T) \times\R),\,T>0,$ as $\eps\to 0$.}
\end{equation*}
Thanks to Lemma \ref{lm:stima-l-10},
\begin{align*}
\norm{\beta^2\eta'(\ueb)\pxx\ueb}^2_{L^2((0,T)\times\R))}&\leq \norm{\eta'}^2_{L^{\infty}(\R)}\beta ^2\int_{0}^{T} \norm{\pxx\ueb(s,\cdot)}^2_{L^2(\R)}ds\\
&\leq \norm{\eta'}^2_{L^{\infty}(\R)}C(T)\eps \to 0.
\end{align*}
Let us  show that
\begin{equation*}
\{I_{4,\,\eps,\,\beta}\}_{\eps,\,\beta >0}\quad\text{is bounded in $L^1((0,T)\times\R),\, T>0$}.
\end{equation*}
Again by Lemma \ref{lm:stima-l-10},
\begin{align*}
&\norm{\beta\eta''(\ueb)\px\ueb\pxx\ueb}_{L^1((0,T)\times\R)}\\
&\qquad\leq \norm{\eta''}_{L^{\infty}(\R)}\beta\int_{0}^{T} \norm{\px\ueb(s,\cdot)\pxx\ueb(s,\cdot)}_{L^1(\R)}ds\\
&\qquad\leq \norm{\eta''}_{L^{\infty}(\R)}C(T).
\end{align*}
We claim that
\begin{equation*}
I_{5,\,\eps,\,\beta}\to0 \quad \text{in $H^{-1}((0,T) \times\R),\,T>0,$ as $\eps\to 0$.}
\end{equation*}
Due to \eqref{eq:beta-eps}, \eqref{eq:u-l-infty}, \eqref{eq:u-l-6} and the H\"older inequality,
\begin{align*}
&\norm{3\beta\eta'(\ueb)\ueb^2\px\ueb}^2_{L^2((0,T)\times\R)}\\
&\qquad\leq 9\norm{\eta'}_{L^{\infty}(\R)}\beta^2\int_{0}^{T}\!\!\!\int_{\R}\ueb^4(\px\ueb)^2 dsdx\\
&\qquad\le 9\norm{\eta'}_{L^{\infty}(\R)}\beta^2\norm{\ueb(t,\cdot)}_{L^{\infty}(\R)}\int_{0}^{T}\!\!\!\int_{\R}\ueb^3(\px\ueb)^2 dsdx\\
&\qquad \le C_{0}\norm{\eta'}_{L^{\infty}(\R)}\beta^{\frac{7}{4}}\int_{0}^{T}\!\!\!\int_{\R}\ueb^3(\px\ueb)^2 dsdx\\
&\qquad\le C_{0}\norm{\eta'}_{L^{\infty}(\R)}\frac{\beta^{\frac{5}{4}}\beta^{\frac{1}{2}}\eps}{\eps}\norm{\ueb}^3_{L^6((0,T)\times\R)}\norm{\px\ueb}^2_{L^4((0,T)\times\R)}\\
&\qquad \le C_{0}T^{\frac{1}{2}}\norm{\eta'}_{L^{\infty}(\R)}\eps^{\frac{1}{2}}\to 0.
\end{align*}
We have that
\begin{equation*}
\{I_{6,\,\eps,\,\beta}\}_{\eps,\,\beta >0}\quad\text{is bounded in $L^1((0,T)\times\R),\, T>0$}.
\end{equation*}
Thanks to \eqref{eq:l-2},
\begin{align*}
&\norm{3\beta\eta''(\ueb)\ueb^2(\px\ueb)^2}_{L^1((0,T)\times\R)}\\
&\qquad\leq 3\norm{\eta''}_{L^{\infty}(\R)}\beta\int_{0}^{T}\!\!\!\int_{\R}\ueb^2(\px\ueb)^2dtdx\\
&\qquad \le C_{0} \norm{\eta''}_{L^{\infty}(\R)}.
\end{align*}
We claim that
\begin{equation*}
I_{7,\,\eps,\,\beta}\to0 \quad \text{in $L^1((0,T) \times\R),\,T>0,$ as $\eps\to 0$.}
\end{equation*}
Due to \eqref{eq:beta-eps}, \eqref{eq:l-2} and \eqref{eq:u-l-infty},
\begin{align*}
&\norm{15\beta\eta'(\ueb)\ueb(\px\ueb)^2}_{L^1((0,T)\times\R)}\\
&\qquad \le 15\norm{\eta'}_{L^{\infty}(\R)}\beta\int_{0}^{T}\!\!\!\int_{\R}\vert \ueb\vert (\px\ueb)^2 dsdx\\
&\qquad \le 15\norm{\eta'}_{L^{\infty}(\R)}\beta \norm{\ueb(t,\cdot)}_{L^{\infty}(\R)} \int_{0}^{T}\!\!\!\int_{\R}(\px\ueb)^2 dsdx\\
&\qquad \le C_{0}\norm{\eta'}_{L^{\infty}(\R)}\beta^{\frac{3}{4}}\int_{0}^{T}\!\!\!\int_{\R}(\px\ueb)^2 dsdx\\
&\qquad \le C_{0}\norm{\eta'}_{L^{\infty}(\R)}\eps^{\frac{3}{2}}\int_{0}^{T}\!\!\!\int_{\R}(\px\ueb)^2 dsdx\\
&\qquad \le C_{0}\norm{\eta'}_{L^{\infty}(\R)}\eps^{\frac{1}{2}}\to 0.
\end{align*}

Therefore, Lemma \ref{lm:1} and the $L^p$ compensated compactness  \cite{SC} give \eqref{eq:con1}.

We conclude by proving that $u$ is a distributional solution of \eqref{eq:Ise}.
Let $ \phi\in C^{\infty}(\R^2)$ be a test function with  compact support. We have to prove that
\begin{equation}
\label{eq:k1}
\int_{0}^{\infty}\!\!\!\!\!\int_{\R}\left(u\pt\phi-\frac{3u^5}{5}\px\phi\right)dtdx +\int_{\R}u_{0}(x)\phi(0,x)dx=0.
\end{equation}
We have that
\begin{align*}
\int_{0}^{\infty}\!\!&\!\!\!\int_{\R}\left(u_{\eps_{n}, \beta_{n}}\pt\phi-\frac{3u^5_{\eps_n, \beta_{n}}}{5}\px\phi\right)dtdx +\int_{\R}u_{0,\eps_n,\beta_n}(x)\phi(0,x)dx\\
&+\eps_{n}\int_{0}^{\infty}\!\!\!\!\!\int_{\R}u_{\eps_{n},\beta_{n}}\pxx\phi dtdx + \eps_n\int_{0}^{\infty}u_{0,\eps_{n},\beta_{n}}(x)\pxx\phi(0,x)dx\\
&+ \beta_n\int_{0}^{\infty}\!\!\!\!\int_{\R}u_{\eps_n,\beta_n}\pxxx\phi dt dx + \beta_n\int_{0}^{\infty}u_{0,\eps_n,\beta_n}(x)\pxxx\phi(0,x)dx\\
=&- 3\beta_n \int_{0}^{\infty}\!\!\!\!\int_{\R}u^2_{\eps_n,\beta_n}\pxx u_{\eps_n,\beta_n}\phi dt dx+9\beta\int_{0}^{\infty}\!\!\!\!\int_{\R}u_{\eps_n,\beta_n}\left(\px u_{\eps_n,\beta_n}\right)^2 \phi dt dx\\
=&15 \beta\int_{0}^{\infty}\!\!\!\!\int_{\R}u_{\eps_n,\beta_n}\left(\px u_{\eps_n,\beta_n}\right)^2 \phi dt dx + 3\beta_n \int_{0}^{\infty}\!\!\!\!\int_{\R}u^2_{\eps_n,\beta_n}\px u_{\eps_n,\beta_n}\px\phi\\
=& 15 \beta\int_{0}^{\infty}\!\!\!\!\int_{\R}u_{\eps_n,\beta_n}\left(\px u_{\eps_n,\beta_n}\right)^2 \phi dt dx-\beta_n\int_{0}^{\infty}\!\!\!\!\int_{\R}u^3_{\eps_n,\beta_n}\pxx\phi dtdx.
\end{align*}
Let us show that
\begin{equation}
\label{eq:54}
15 \beta\int_{0}^{\infty}\!\!\!\!\int_{\R}u_{\eps_n,\beta_n}\left(\px u_{\eps_n,\beta_n}\right)^2 \phi dt dx\to 0.
\end{equation}
Due to \eqref{eq:beta-eps}, \eqref{eq:l-2} and \eqref{eq:u-l-infty},
\begin{align*}
&15 \beta\left\vert\int_{0}^{\infty}\!\!\!\!\int_{\R}u_{\eps_n,\beta_n}\left(\px u_{\eps_n,\beta_n}\right)^2 \phi dt dx\right\vert\\
&\quad \le 15\beta\int_{0}^{\infty}\!\!\!\!\int_{\R}\left\vert u_{\eps_n,\beta_n}\right\vert\left(\px u_{\eps_n,\beta_n}\right)^2 \vert\phi\vert dt dx\\
&\quad \le 15\beta\norm{u_{\eps_n,\beta_n}}_{L^{\infty}(\R^{+}\times\R)}\norm{\phi}_{L^{\infty}(\R^{+}\times\R)}\int_{0}^{\infty}\!\!\!\!\int_{\R}\left(\px u_{\eps_n,\beta_n}\right)^2dt dx\\
&\quad \le C_{0}\beta^{\frac{3}{4}}\norm{\phi}_{L^{\infty}(\R^{+}\times\R)}\int_{0}^{\infty}\!\!\!\!\int_{\R}\left(\px u_{\eps_n,\beta_n}\right)^2dt dx\\
&\quad \le C_{0}\eps^{\frac{3}{2}}\norm{\phi}_{L^{\infty}(\R^{+}\times\R)}\int_{0}^{\infty}\!\!\!\!\int_{\R}\left(\px u_{\eps_n,\beta_n}\right)^2dt dx\\
&\quad \le C_{0}\eps^{\frac{1}{2}}\norm{\phi}_{L^{\infty}(\R^{+}\times\R)}\to 0,
\end{align*}
that is \eqref{eq:54}.

We prove that
\begin{equation}
\label{eq:55}
\beta_n\int_{0}^{\infty}\!\!\!\!\int_{\R}u^3_{\eps_n,\beta_n}\pxx\phi dtdx\to 0.
\end{equation}
Thanks to \eqref{eq:u-l-6} and the H\"older inequality,
\begin{align*}
&\beta_n\left\vert\int_{0}^{\infty}\!\!\!\!\int_{\R}u^3_{\eps_n,\beta_n}\pxx\phi dtdx\right\vert\le \beta_{n} \int_{0}^{\infty}\!\!\!\!\int_{\R}\vert u^3_{\eps_n,\beta_n}\vert\vert\pxx\phi\vert dtdx\\
&\qquad\le \beta_{n} \norm{u_{\eps_n,\beta_n}}^3_{L^6(\supp(\pxx\phi))}\norm{\pxx\phi}_{L^2(\supp(\phi))}\\
&\qquad\le \beta_{n}\norm{u_{\eps_n,\beta_n}}^3_{L^6((0,T)\times\R)}\norm{\pxx\phi}_{L^2((0,T)\times\R)}\\
&\qquad\le \beta_{n}\norm{\pxx\phi}_{L^2((0,T)\times\R)}C_{0}T^{\frac{1}{2}}\to0,
\end{align*}
that is \eqref{eq:55}.

Therefore, \eqref{eq:k1} follows from \eqref{eq:u0eps}, \eqref{eq:con1}, \eqref{eq:54} and \eqref{eq:55}.
\end{proof}
Following \cite{LN}, we prove the following result.
\begin{lemma}
\label{lm:entropy-solution}
Assume that \eqref{eq:u-0}, \eqref{eq:u0eps}, and \eqref{eq:beta-eps-1} hold. Then,
\begin{equation}
\label{eq:con3}
 u_{\eps_n, \beta_n}\to u \quad  \textrm{in} \quad  L^{p}_{loc}((0,\infty)\times\R),\quad \textrm{for each} \quad 1\le p <10,\\
\end{equation}
where $u$ is  the unique entropy solution of \eqref{eq:con-law}.
\end{lemma}
\begin{proof}
Let us consider a compactly supported entropy--entropy flux pair $(\eta, q)$. Multiplying \eqref{eq:Iseepsw} by $\eta'(\ueb)$, we obtain that
\begin{align*}
\pt\eta(\ueb) + \px q(\ueb) =&\eps \eta'(\ueb) \pxx\ueb + \beta \eta'(\ueb) \pxxx\ueb + \\
&+3\beta\eta'(\ueb)\ueb^2\pxx\ueb -9\beta\eta'(\ueb)\ueb(\px\ueb)^2\\
=& I_{1,\,\eps,\,\beta}+I_{2,\,\eps,\,\beta}+ I_{3,\,\eps,\,\beta} + I_{4,\,\eps,\,\beta}\\
&+I_{5,\,\eps,\,\beta}+I_{6,\,\eps,\,\beta}+I_{7,\,\eps,\,\beta},
\end{align*}
where $I_{1,\,\eps,\,\beta},\,I_{2,\,\eps,\,\beta},\, I_{3,\,\eps,\,\beta} ,\, I_{4,\,\eps,\,\beta}$, $I_{5,\,\eps,\,\beta}$, $I_{6,\,\eps,\,\beta}$, $I_{7,\,\eps,\,\beta}$  are defined in \eqref{eq:12000}.

Arguing as \cite[Lemma $3.3$]{Cd2}, we obtain that $I_{1,\,\eps,\,\beta}\to0$ in $H^{-1}((0,T) \times\R)$, $\{I_{2,\,\eps,\,\beta}\}_{\eps,\beta >0}$ is bounded in $L^1((0,T)\times\R)$, $I_{3,\,\eps,\,\beta}\to0$ in $H^{-1}((0,T) \times\R)$, $I_{5,\,\eps,\,\beta}\to0$ in $H^{-1}((0,T) \times\R)$ and $I_{7,\,\eps,\,\beta}\to0$ in $L^1((0,T)\times\R)$.

Let us  show that
\begin{equation*}
I_{4,\,\eps,\,\beta}\to 0\quad\text{in $L^1((0,T)\times\R),\, T>0$ as $\eps\to 0$}.
\end{equation*}
Due to \eqref{eq:beta-eps-1}, \eqref{eq:l-2}, \eqref{eq:l-10}, and the H\"older inequality,
\begin{align*}
&\norm{\beta\eta''(\ueb)\px\ueb\pxx\ueb}_{L^1((0,T)\times\R)}\\
&\qquad\leq \norm{\eta''}_{L^{\infty}(\R)}\beta\int_{0}^{T} \norm{\px\ueb(s,\cdot)\pxx\ueb(s,\cdot)}_{L^1(\R)}ds\\
&\qquad\leq C_{0}\norm{\eta''}_{L^{\infty}(\R)}\eps^{\alpha}\eps^2\int_{0}^{T} \norm{\px\ueb(s,\cdot)\pxx\ueb(s,\cdot)}_{L^1(\R)}ds\\
&\qquad= C_{0}\norm{\eta''}_{L^{\infty}(\R)}\eps^{\alpha}\int_{0}^{T}\!\!\!\int_{\R}\left\vert\eps^{\frac{1}{2}} \px\ueb\right\vert\left\vert\eps^{\frac{3}{2}}\pxx\ueb\right\vert dsdx\\
&\qquad\le C_{0}\norm{\eta''}_{L^{\infty}(\R)}\eps^{\alpha}\eps^{\frac{1}{2}}\norm{\px\ueb}_{L^2((0,T)\times\R)}\eps^{\frac{3}{2}}\norm{\pxx\ueb}_{L^2((0,T)\times\R)}\\
&\qquad \le C(T)\norm{\eta''}_{L^{\infty}(\R)}\eps^{\alpha}\to0.
\end{align*}
We have that
\begin{equation*}
I_{6,\,\eps,\,\beta}\to 0\quad\text{is in $L^1((0,T)\times\R),\, T>0$ as $\eps\to 0$}.
\end{equation*}
Thanks to \eqref{eq:beta-eps-1}, \eqref{eq:u-l-4}, \eqref{eq:l-10} and the H\"older inequality,
\begin{align*}
&\norm{3\beta\eta''(\ueb)\ueb^2(\px\ueb)^2,}_{L^1((0,T)\times\R)}\\
&\qquad\leq 3\norm{\eta''}_{L^{\infty}(\R)}\beta\int_{0}^{T}\!\!\!\int_{\R}\ueb^2(\px\ueb)^2dsdx\\
&\qquad \le 3\norm{\eta''}_{L^{\infty}(\R)}\beta\norm{\ueb}^2_{L^4((0,T)\times\R)}\norm{\px\ueb}^2_{L^4((0,T)\times\R)}.\\
&\qquad \le C_{0}T^{\frac{1}{2}}\norm{\eta''}_{L^{\infty}(\R)}\frac{\beta\eps}{\eps}\norm{\px\ueb}^2_{L^4((0,T)\times\R)}\\
&\qquad\le C(T)\norm{\eta''}_{L^{\infty}(\R)}\frac{\beta^{\frac{1}{2}}}{\eps}\\
&\qquad = C(T)\norm{\eta''}_{L^{\infty}(\R)}\eps^{\frac{\alpha}{2}}\to 0.
\end{align*}
Therefore, Lemma \ref{lm:1} gives \eqref{eq:con3}.

We conclude by proving that $u$ is the unique entropy solution of \eqref{eq:con-law}.
Let us consider a compactly supported entropy--entropy flux pair $(\eta, q)$, and $\phi\in C^{\infty}_{c}((0,\infty)\times\R)$ a non--negative function. We have to prove that
\begin{equation}
\label{eq:u-entropy-solution}
\int_{0}^{\infty}\!\!\!\!\!\int_{\R}(\pt\eta(u)+ \px q(u))\phi dtdx \le0.
\end{equation}
We have that
\begin{align*}
&\int_{0}^{\infty}\!\!\!\!\!\int_{\R}(\px\eta(u_{\eps_{n},\,\beta_{n}})+\px q(u_{\eps_{n},\,\beta_{n}}))\phi dtdx\\
&\qquad=\eps_{n}\int_{0}^{\infty}\!\!\!\!\!\int_{\R}\px(\eta'(u_{\eps_{n},\,\beta_{n}})\px u_{\eps_{n},\,\beta_{n}})\phi dtdx -\eps_{n}\int_{0}^{\infty}\!\!\!\!\!\int_{\R} \eta''(u_{\eps_{n},\,\beta_{n}})(\px u_{\eps_{n},\,\beta_{n}})^2\phi dtdx\\
&\qquad\quad -\beta_{n}\int_{0}^{\infty}\!\!\!\!\!\int_{\R}\px(\eta'(u_{\eps_{n},\,\beta_{n}})\pxx u_{\eps_{n},\,\beta_{n}})\phi dtdx\\
&\qquad\quad +\beta_{n}\int_{0}^{\infty}\!\!\!\!\!\int_{\R}\eta''(u_{\eps_{n},\,\beta_{n}})\px u_{\eps_{n},\,\beta_{n}}\pxx u_{\eps_{n},\,\beta_{n}}\phi dtdx\\
&\qquad\quad +3\beta_{n}\int_{0}^{\infty}\!\!\!\!\!\int_{\R}\px(\eta'(u_{\eps_{n},\,\beta_{n}})u^2_{\eps_{n},\,\beta_{n}}\px u_{\eps_{n},\,\beta_{n}} )\phi dtdx\\
&\qquad\quad -3\beta_{n}\int_{0}^{\infty}\!\!\!\!\!\int_{\R}\eta''(u_{\eps_{n},\,\beta_{n}})u^2_{\eps_{n},\,\beta_{n}}(\px u_{\eps_{n},\,\beta_{n}})^2\phi dtdx\\
&\qquad\quad -15\beta_{n}\int_{0}^{\infty}\!\!\!\!\!\int_{\R}\eta'(u_{\eps_{n},\,\beta_{n}})u_{\eps_{n},\,\beta_{n}}(\px u_{\eps_{n},\,\beta_{n}})^2\phi dtdx\\
&\qquad \le - \eps_{n}\int_{0}^{\infty}\!\!\!\!\!\int_{\R}\eta'(u_{\eps_{n},\,\beta_{n}})\px u_{\eps_{n},\,\beta_{n}}\px\phi dtdx + \beta_{n}\int_{0}^{\infty}\!\!\!\!\!\int_{\R}\eta'(u_{\eps_{n},\,\beta_{n}})\pxx u_{\eps_{n},\,\beta_{n}}\px\phi dtdx\\
&\qquad\quad +\beta_{n}\int_{0}^{\infty}\!\!\!\!\!\int_{\R}\eta''(u_{\eps_{n},\,\beta_{n}})\px u_{\eps_{n},\,\beta_{n}}\pxx u_{\eps_{n},\,\beta_{n}}\phi dtdx\\
&\qquad\quad -3\beta_{n}\int_{0}^{\infty}\!\!\!\!\!\int_{\R}\eta'(u_{\eps_{n},\,\beta_{n}})u^2_{\eps_{n},\,\beta_{n}}\px u_{\eps_{n},\,\beta_{n}}\px\phi dtdx\\
&\qquad\quad-3\beta_{n}\int_{0}^{\infty}\!\!\!\!\!\int_{\R}\eta''(u_{\eps_{n},\,\beta_{n}})u^2_{\eps_{n},\,\beta_{n}}(\px u_{\eps_{n},\,\beta_{n}})^2\phi dtdx\\
&\qquad\quad -15\beta_{n}\int_{0}^{\infty}\!\!\!\!\!\int_{\R}\eta'(u_{\eps_{n},\,\beta_{n}})u_{\eps_{n},\,\beta_{n}}(\px u_{\eps_{n},\,\beta_{n}})^2\phi dtdx\\
&\qquad \le  \eps_{n}\int_{0}^{\infty}\!\!\!\!\!\int_{\R}\vert\eta'(u_{\eps_{n},\,\beta_{n}})\vert\vert\px u_{\eps_{n},\,\beta_{n}}\vert\vert\px\phi\vert dtdx\\
&\qquad\quad+\beta_{n}\int_{0}^{\infty}\!\!\!\!\!\int_{\R}\vert\eta'(u_{\eps_{n},\,\beta_{n}})\vert\vert\pxx u_{\eps_{n},\,\beta_{n}}\vert\vert\px\phi\vert dtdx\\
&\qquad\quad +\beta_{n}\int_{0}^{\infty}\!\!\!\!\!\int_{\R}\vert\eta'(u_{\eps_{n},\,\beta_{n}})\vert\vert\pxx u_{\eps_{n},\,\beta_{n}}\vert\vert\px\phi\vert dtdx\\
&\qquad\quad +3\beta_{n}\int_{0}^{\infty}\!\!\!\!\!\int_{\R}\vert\eta'(u_{\eps_{n},\,\beta_{n}})\vert u^2_{\eps_{n},\,\beta_{n}}\vert\px u_{\eps_{n},\,\beta_{n}}\vert\vert\px\phi\vert dtdx\\
&\qquad\quad+3\beta_{n}\int_{0}^{\infty}\!\!\!\!\!\int_{\R}\vert\eta''(u_{\eps_{n},\,\beta_{n}})\vert u^2_{\eps_{n},\,\beta_{n}}(\px u_{\eps_{n},\,\beta_{n}})^2\vert\phi\vert dtdx\\
&\qquad\quad+ 15\beta_{n}\int_{0}^{\infty}\!\!\!\!\!\int_{\R}\vert\eta'(u_{\eps_{n},\,\beta_{n}})\vert \vert u_{\eps_{n},\,\beta_{n}}\vert(\px u_{\eps_{n},\,\beta_{n}})^2\vert\phi\vert dtdx\\
&\qquad\le  \eps_{n} \norm{\eta'}_{L^{\infty}(\R)}\norm{\px u_{\eps_{n},\,\beta_{n}}}_{L^2(\supp(\px\phi))}\norm{\px\phi}_{L^2(\supp(\px\phi))}\\
&\qquad\quad+ \beta_{n} \norm{\eta'}_{L^{\infty}(\R)}\norm{\pxx u_{\eps_{n},\,\beta_{n}}}_{L^2(\supp(\px\phi))}\norm{\px\phi}_{L^2(\supp(\px\phi))}\\
&\qquad\quad +\beta_{n} \norm{\eta''}_{L^{\infty}(\R)}\norm{\phi}_{L^{\infty}(\R^{+}\times\R)}\norm{\px u_{\eps_{n},\,\beta_{n}}\pxx u_{\eps_{n},\,\beta_{n}}}_{L^1(\supp(\px\phi))}\\
&\qquad\quad +3\beta_{n}\norm{\eta'}_{L^{\infty}(\R)}\int_{0}^{\infty}\!\!\!\!\!\int_{\R} u^2_{\eps_{n},\,\beta_{n}}\vert\px u_{\eps_{n},\,\beta_{n}}\vert\vert\px\phi\vert dtdx\\
&\qquad\quad +3\beta_{n}\norm{\eta''}_{L^{\infty}(\R)} \norm{\phi}_{L^{\infty}(\R^{+}\times\R)}\norm{u^2_{\eps_{n},\,\beta_{n}}(\px u_{\eps_{n},\,\beta_{n}})^2}_{L^1(\supp(\phi))}\\
&\qquad\quad+15\beta_{n}\norm{\eta'}_{L^{\infty}(\R)}\norm{\phi}_{L^{\infty}(\R^{+}\times\R)}\norm{u_{\eps_{n},\,\beta_{n}}(\px u_{\eps_{n},\,\beta_{n}})^2}_{L^1(\supp(\phi))}\\
&\qquad\le  \eps_{n} \norm{\eta'}_{L^{\infty}(\R)}\norm{\px u_{\eps_{n},\,\beta_{n}}}_{L^2((0,T)\times\R)}\norm{\px\phi}_{L^2((0,T)\times\R)}\\
&\qquad\quad+ \beta_{n} \norm{\eta'}_{L^{\infty}(\R)}\norm{\pxx u_{\eps_{n},\,\beta_{n}}}_{L^2((0,T)\times\R)}\norm{\px\phi}_{L^2((0,T)\times\R)}\\
&\qquad\quad+\beta_{n} \norm{\eta''}_{L^{\infty}(\R)}\norm{\phi}_{L^{\infty}(\R^{+}\times\R)}\norm{\px u_{\eps_{n},\,\beta_{n}}\pxx u_{\eps_{n},\,\beta_{n}}}_{L^1((0,T)\times\R)}\\
&\qquad\quad +3\beta_{n}\norm{\eta'}_{L^{\infty}(\R)}\int_{0}^{\infty}\!\!\!\!\!\int_{\R} u^2_{\eps_{n},\,\beta_{n}}\vert\px u_{\eps_{n},\,\beta_{n}}\vert\vert\px\phi\vert dtdx\\
&\qquad\quad+3\beta_{n}\norm{\eta''}_{L^{\infty}(\R)} \norm{\phi}_{L^{\infty}(\R^{+}\times\R)}\norm{u^2_{\eps_{n},\,\beta_{n}}(\px u_{\eps_{n},\,\beta_{n}})^2}_{L^1((0,T)\times\R)}\\
&\qquad\quad +15\beta_{n}\norm{\eta'}_{L^{\infty}(\R)}\norm{\phi}_{L^{\infty}(\R^{+}\times\R)}\norm{u_{\eps_{n},\,\beta_{n}}(\px u_{\eps_{n},\,\beta_{n}})^2}_{L^1((0,T)\times\R)}.
\end{align*}
We show that
\begin{equation}
\label{eq:789}
3\beta_{n}\norm{\eta'}_{L^{\infty}(\R)}\int_{0}^{\infty}\!\!\!\!\!\int_{\R} u^2_{\eps_{n},\,\beta_{n}}\vert\px u_{\eps_{n},\,\beta_{n}}\vert\vert\px\phi\vert dtdx\to 0.
\end{equation}
Due to \eqref{eq:beta-eps-1}, \eqref{eq:l-2}, \eqref{eq:u-l-infty} and the H\"older inequality,
\begin{align*}
&3\beta_{n}\norm{\eta'}_{L^{\infty}(\R)}\int_{0}^{\infty}\!\!\!\!\!\int_{\R} u^2_{\eps_{n},\,\beta_{n}}\vert\px u_{\eps_{n},\,\beta_{n}}\vert\vert\px\phi\vert dtdx\\
&\qquad\le 3\norm{\eta'}_{L^{\infty}(\R)}\beta_{n}\norm{u_{\eps_{n},\,\beta_{n}}(t,\cdot)}^2_{L^{\infty}(\R)}\\
&\qquad\quad\cdot\norm{\px u_{\eps_{n},\,\beta_{n}}}_{L^2(\supp(\px\phi)}\norm{\px\phi}_{L^2(\supp(\phi))}\\
&\qquad \le C_{0}\norm{\eta'}_{L^{\infty}(\R)}\beta^{\frac{1}{2}}_{n}\norm{\px u_{\eps_{n},\,\beta_{n}}}_{L^2((0,T)\times\R)}\norm{\px\phi}_{L^2((0,T)\times\R)}\\
&\qquad\le C_0\norm{\eta'}_{L^{\infty}(\R)}\norm{\px\phi}_{L^2((0,T)\times\R)}\eps^{\frac{\alpha}{2}}_{n}\eps_{n}\norm{\px u_{\eps_{n},\,\beta_{n}}}_{L^2((0,T)\times\R)}\\
&\qquad\le C_{0}\norm{\eta'}_{L^{\infty}(\R)}\norm{\px\phi}_{L^2((0,T)\times\R)}\eps^{\frac{\alpha+1}{2}}_{n}\to0,
\end{align*}
that is \eqref{eq:789}.

We claim that
\begin{equation}
\label{eq:790}
3\beta_{n}\norm{\eta''}_{L^{\infty}(\R)} \norm{\phi}_{L^{\infty}(\R^{+}\times\R)}\norm{u^2_{\eps_{n},\,\beta_{n}}(\px u_{\eps_{n},\,\beta_{n}})^2}_{L^1((0,T)\times\R)}\to 0.
\end{equation}
Due to \eqref{eq:beta-eps-1},  \eqref{eq:l-2} and \eqref{eq:u-l-infty},
\begin{align*}
& 3\beta_{n}\norm{\eta''}_{L^{\infty}(\R)} \norm{\phi}_{L^{\infty}(\R^{+}\times\R)}\int_{0}^{T}\!\!\!\int_{\R} u^2_{\eps_{n},\,\beta_{n}}(\px u_{\eps_{n},\,\beta_{n}})^2dsdx\\
&\qquad\le 3\norm{\eta''}_{L^{\infty}(\R)} \norm{\phi}_{L^{\infty}(\R^{+}\times\R)}\beta_{n}\norm{u_{\eps_{n},\,\beta_{n}}(t,\cdot)}^2_{L^{\infty}(\R)}\\               &\qquad\quad\cdot\int_{0}^{T}\!\!\!\int_{\R} (\px u_{\eps_{n},\,\beta_{n}})^2dsdx\\
&\qquad\le C_{0}\norm{\eta''}_{L^{\infty}(\R)} \norm{\phi}_{L^{\infty}(\R^{+}\times\R)}\beta^{\frac{1}{2}}_{n}\int_{0}^{T}\!\!\!\int_{\R} (\px u_{\eps_{n},\,\beta_{n}})^2dsdx\\
&\qquad\le C_{0}\norm{\eta''}_{L^{\infty}(\R)} \norm{\phi}_{L^{\infty}(\R^{+}\times\R)}\eps^{\frac{\alpha}{2}}_{n}\eps_{n}\int_{0}^{T}\!\!\!\int_{\R} (\px u_{\eps_{n},\,\beta_{n}})^2dsdx\\
&\qquad \le C_{0}\norm{\eta''}_{L^{\infty}(\R)} \norm{\phi}_{L^{\infty}(\R^{+}\times\R)}\eps^{\frac{\alpha}{2}}\to 0,
\end{align*}
that is \eqref{eq:790}.

We have
\begin{equation}
\label{eq:791}
15\beta_{n}\norm{\eta'}_{L^{\infty}(\R)}\norm{\phi}_{L^{\infty}(\R^{+}\times\R)}\norm{u_{\eps_{n},\,\beta_{n}}(\px u_{\eps_{n},\,\beta_{n}})^2}_{L^1((0,T)\times\R)}\to 0.
\end{equation}
Again by \eqref{eq:beta-eps-1},  \eqref{eq:l-2} and \eqref{eq:u-l-infty},
\begin{align*}
&15\beta_{n}\norm{\eta'}_{L^{\infty}(\R)}\norm{\phi}_{L^{\infty}(\R^{+}\times\R)}\int_{0}^{T}\!\!\!\int_{\R}u_{\eps_{n},\,\beta_{n}}(\px u_{\eps_{n},\,\beta_{n}})^2dsdx\\
&\qquad\le 15\norm{\eta'}_{L^{\infty}(\R)}\norm{\phi}_{L^{\infty}(\R^{+}\times\R)}\beta_{n}\norm{u_{\eps_{n},\,\beta_{n}}(t,\cdot)}_{L^{\infty}(\R)}\\
&\qquad\quad\cdot\int_{0}^{T}\!\!\!\int_{\R} (\px u_{\eps_{n},\,\beta_{n}})^2dsdx\\
&\qquad\le C_{0}\norm{\eta'}_{L^{\infty}(\R)}\norm{\phi}_{L^{\infty}(\R^{+}\times\R)}\beta^{\frac{3}{4}}_{n}\int_{0}^{T}\!\!\!\int_{\R} (\px u_{\eps_{n},\,\beta_{n}})^2dsdx\\
&\qquad\le C_{0}\norm{\eta'}_{L^{\infty}(\R)}\norm{\phi}_{L^{\infty}(\R^{+}\times\R)}\eps^{\frac{2+3\alpha}{4}}_{n}\eps_{n}\int_{0}^{T}\!\!\!\int_{\R} (\px u_{\eps_{n},\,\beta_{n}})^2dsdx\\
&\qquad \le C_{0}\norm{\eta'}_{L^{\infty}(\R)}\norm{\phi}_{L^{\infty}(\R^{+}\times\R)}\eps^{\frac{2+3\alpha}{4}}_{n}\to 0,
\end{align*}
that is \eqref{eq:791}.

\eqref{eq:u-entropy-solution} follows from \eqref{eq:beta-eps-1}, \eqref{eq:con3}, \eqref{eq:789}, \eqref{eq:790}, \eqref{eq:791},  Lemmas \ref{lm:l-2} and \ref{lm:stima-l-10}.
\end{proof}
\begin{proof}[Proof of Theorem \ref{th:main}]
Lemma \ref{lm:dist-solution} gives $i)$ and $ii)$, while $iii)$ follows from Lemma \ref{lm:entropy-solution}.
\end{proof}

\end{document}